\documentclass[12pt]{article}
\usepackage{amsmath}
\usepackage{amsfonts}
\usepackage{amssymb}
\usepackage{graphicx}%
\providecommand{\U}[1]{\protect \rule{.1in}{.1in}}
\newtheorem{theorem}{Theorem}[section]

\newtheorem{definition}[theorem]{Definition}

\newtheorem{lemma}{Lemma}[section]

\newenvironment{proof}[1][Proof]{\noindent \textbf{#1.} }{\  $\Box$}

\title{$L^p$-Solutions of Reflected Backward Doubly Stochastic Differential Equations}

\author{Wen L\"{u}\footnote{Support by the National Basic
Research Program of China (973 Program) grant No. 2007CB814900 and
The Youth Fund of Yantai University (SX08Z9).} \footnote{\emph{Email address:} llcxw@163.com}\\
        \footnotesize{  School of Mathematics, Shandong University, Jinan, \small{250100}, China
 }\\  \footnotesize{School of Mathematics, Yantai University, Yantai 264005, China
 } }

\begin{document}
\date{}
\maketitle
\begin{abstract}
 In this paper,  we deal with a class of one-dimensional
reflected backward doubly stochastic differential equations with one
continuous lower barrier.   We derive the existence and uniqueness
of $L^p$-solutions for those equations with Lipschitz coefficients.
\end{abstract}



\section{Introduction}
The general nonlinear case backward stochastic differential equation
(BSDE in short) was first introduced in Pardoux and Peng (1990), who
proved the existence and uniqueness result  when the coefficient
 is  Lipschitz. El Karoui et al. (1997a) introduced the notion of one barrier
reflected BSDE , which is actually a backward equation but the
solution is forced to stay above a given barrier. This type of BSDEs
is motivated by pricing American options (see El Karoui et al.
(1997b)) and studying the mixed game problems  (see e.g.
Cvitani\'{c} and Karatzas (1996), Hamad\`{e}ne and Lepeltier
(2000)).  In order to give a probabilistic representation for a
class of quasilinear stochastic partial differential equations,
Pardoux and Peng (1992) first considered a class of backward doubly
stochastic differential equations (BDSDEs) with two different
directions of stochastic integrals.

However in most of the previous works, solutions are taken in $L^2$
space or in $L^p, p>2$. This limits the scope for several
applications. To correct this shortcoming, El Karoui et al.  (1997c)
obtained the first result on the existence and uniqueness of
solution in $L^p$, $p\in(1,2)$ with a Lipschitz coefficient. Briand
et al. (2003) generalized this result to the BSDEs with monotone
coefficients. Following this way, Aman (2009) considered the
$L^p$-solutions of BDSDEs with a monotone coefficient. Moreover,
Hamad\`{e}ne and Popier (2008) established the existence and
uniqueness of the $L^p$-solutions of BSDEs with reflection having a
Lipschitz coefficient.

More recently, Bahalai et al. (2009) obtained the existence and
uniqueness of solution for BDSDEs with one continuous lower barrier,
having a continuous coefficient. Motivated by above works, the
purpose of this paper is to prove the existence and uniqueness of
$L^p$-solutions for reflected BDSDEs with Lipschitz coefficients.

The rest of the paper is organized as follows.  In Section 2, we
introduce some  preliminaries including some spaces.  With the help
of some a priori estimates, Section 3 is devoted to the existence
and uniqueness of $L^p$-solutions for those equations.

 \section{Preliminaries}
Let $T>0$ a fixed real number. Let $\{W_t\}_{t\geq 0},
\{B_t\}_{t\geq 0}$ be two mutually independent standard Brownian
motions defined on a complete probability space $(\Omega,
\mathcal{F}, {\bf P})$ with values in ${\bf R}^d$ and ${\bf R}$,
respectively. For $t\in[0,T]$, we define
$$\mathcal{F}_t=\mathcal{F}_t^W\vee \mathcal{F}_{t,T}^B,$$
where $\mathcal{F}_t^W=\sigma\{W_s, 0\leq s\leq t\},
\mathcal{F}_{t,T}^B=\sigma\{B_s-B_t, t\leq s\leq T\}$ completed with
the  ${\bf P}$-null sets. We note that the collection $\{\mathcal
{F}_t; t\in[0, T]\}$ is neither increasing nor decreasing, so it
does not constitute a classical filtration.  The Euclidean norm of a
vector $y\in {\bf R}^n$ will be
 defined by $|y|$.

 Throughout the  paper, we always assume that $p\in(1,2)$. Now, let's introduce the following spaces :\\
$\mathcal{M}^p_d=\{\psi: [0, T]\times\Omega\rightarrow{\bf R}^d,$
predictable, such that ${\bf
E}[(\int_0^T|\psi_s|^2ds)^{\frac{p}{2}}]<\infty\}$;
\\
$\mathcal{S}^p=\{\psi: [0, T]\times\Omega\rightarrow{\bf R},$
progressively measurable, s.t. ${\bf E}
(\sup_{t\in[0,T]}|\psi_t|^p)<\infty\}$;
\\
$\mathcal{S}^p_{ci}=\{A: [0, T]\times\Omega\rightarrow{\bf R}_+,$
continuous, increasing, s.t. $A_0=0$ and ${\bf E}|A_T|^p<\infty\}$.

  The object in this paper is the following reflected BDSDE:
\begin{eqnarray}\label{bsde:1}
\left\{ \begin{array}{l@{\quad \quad}r}Y_t=\xi+\int_t^T f(s, Y_s,
Z_s)ds+\int_t^T g(s, Y_s, Z_s)dB_s\\\quad\qquad+K_T-K_t-\int_t^T
Z_sdW_s,\\
Y_t\geq L_t, \; 0\leq t\leq T \;\mbox{a.s. and}\;
\int_0^T(Y_t-L_t)dK_t=0, \,\mbox{a.s.}
\end{array}\right.
\end{eqnarray}
where the $dW$ is a standard forward It\^{o} integral  and the $dB$
is a backward It\^{o} integral.

On the items $\xi, f, g$ and $L$, we make the following assumptions:
\\
(\textbf{H1}) The terminal condition $\xi: \Omega\rightarrow{\bf R},
\mathcal
{F}_T$-measurable such that ${\bf E}|\xi|^p<\infty$;\\
(\textbf{H2}) the functions $f,\, g: [0,T]\times\Omega\times {\bf
R}\times {\bf R}^d\rightarrow {\bf R}$ are jointly measurable and
satisfy:

  (i) ${\bf E}[(\int_0^T|f_s^0|^2ds)^{\frac{p}{2}}]<\infty,\; {\bf
  E}[(\int_0^T|g_s^0|^2ds)^{\frac{p}{2}}]<\infty$,\\
\indent\quad where $f_s^0=:f(s,0,0), g_s^0=:g(s,0,0)$;

  (ii) $\forall t\in[0,T],
(y_1,z_1),(y_2,z_2)\in{\bf R}\times{\bf R}^d$, there exist
  constants $C>0$\\
\indent\quad and $0<\alpha<1$ such that
\begin{eqnarray*}\label{lip:1}
|f(t,y_1,z_1)-f(t,y_2,z_2)|&\leq& C(|y_1-y_2|+
|z_1-z_2|),\\
|g(t,y_1,z_1)-g(t,y_2,z_2)|^2&\leq& C|y_1-y_2|^2+ \alpha|z_1-z_2|^2;
\end{eqnarray*}
 (\textbf{H3}) The barrier  $\{L_t, t\in[0, T]\}$ is a real valued
progressively measurable process such that $E(\sup_{0\leq t\leq
T}(L_t^+)^p)<\infty$ and $L_T\leq \xi$ a.s..

Let's  give the notion of $L^p$-solution of reflected BDSDE
(\ref{bsde:1}).
\begin{definition} An $L^p$-solution of the reflected BDSDE (\ref{bsde:1}) is a triple of progressively measurable processes $(Y,Z,K)$
 satisfying (\ref{bsde:1}) such that  $(Y,Z,K)\in \mathcal{S}^p\times\mathcal{M}^p_d\times\mathcal{S}^p_{ci}$.
\end{definition}

The following lemma  is a slight generalization of Corollary 2.3 in
Briand et al. (2003).
\begin{lemma}\label{lemma:0}
Let $(Y,Z)\in\mathcal{S}^p \times\mathcal{M}^p_d$ is a solution of
the following BDSDE :
\begin{eqnarray*}
&&|Y_t|=\xi+\int_t^T\widetilde{f}(s,Y_s,Z_s)ds
+\int_t^T\widetilde{g}(s,Y_s,Z_s)dB_s+A_T-A_t-\int_t^TZ_sdW_s,
\end{eqnarray*}
where:

 (i) $\widetilde{f}$ and $\widetilde{g}$ are functions which satisfy the
   assumptions as $f$ and $g$,

(ii) ${\bf P}$-a.s. the process $(A_t)_{t\in[0,T]}$ is of bounded
variation type.\\
Then for any $0\leq t\leq u\leq T$, we have
\begin{eqnarray*}
&&|Y_t|^p+c(p)\int_t^T|Y_s|^{p-2}{\bf 1}_{\{Y_s\neq 0\}}|Z_s|^2ds
\\&\leq&
|Y_u|^p+p\int_t^T|Y_s|^{p-1}\widehat{Y_s}dA_s+p\int_t^T|Y_s|^{p-1}\widetilde{f}(s,
Y_s, Z_s)ds
\\&&
+c(p)\int_t^T|Y_s|^{p-2}{\bf 1}_{\{Y_s\neq
0\}}|\widetilde{g}(s,Y_s,Z_s)|^2ds
\\&&
+p\int_t^T|Y_s|^{p-1} \widehat{Y_s}
\widetilde{g}(s,Y_s,Z_s)dB_s-p\int_t^T|Y_s|^{p-1}\widehat{Y_s}Z_sdW_s,
\end{eqnarray*}
where $c(p)=\frac{p(p-1)}{2}$ and $\widehat{y}=\frac{y}{|y|}{\bf
1}_{\{y\neq 0\}}$.
\end{lemma}

\section{Main results}

\subsection{A priori estimates}
In order to obtain the existence and uniqueness result for solution
of the reflected BDSDE (\ref{bsde:1}), we first provide some a
priori estimates of solution  of (\ref{bsde:1}).

 In what follows,
$d, d_1, d_2,\cdots$ will be denoted as a constant whose value
depending only on $C, \alpha, p$ and possibly $T$. We also denote by
$\theta_1, \theta_2,\cdots$ the constants which taking  value in
$(0,\infty)$ arbitrarily.
\begin{lemma}\label{lemma:1}
Let the assumptions (H1)-(H3) hold and let $(Y,Z, K )$ be a solution
of the reflected BDSDE (\ref{bsde:1}). If $Y\in\mathcal {S}^p$ then
$Z\in\mathcal {M}^p_d$ and  there exists a constant $d>0$ such that
$$
{\bf E}\left[\left(\int_0^T |Z_s|^2ds\right)^{\frac{p}{2}}\right]
\leq d\,{\bf E}\left[\sup_{t\in[0,T]}|Y_t|^p+\left(\int_0^T
|f_s^0|^2ds\right)^{\frac{p}{2}}+\left(\int_0^T
|g_s^0|^2ds\right)^{\frac{p}{2}}\right].
$$
\end{lemma}
\begin{proof}
 For each integer $n\geq 0$, let's define the stopping
time
\begin{eqnarray*}
\tau_n=\inf\{t\in[0,T], \;\int_0^t|Z_s|^2ds\geq n\}\wedge T.
\end{eqnarray*}
Let $a\in{\bf R}$,  using It\^{o}'s formula and assumption (H2), we get %%%%%%%%%%%%%%12.8
\begin{eqnarray*}
&&|Y_0|^2+\int_0^{\tau_n}e^{as}|Z_s|^2ds\\
&=&e^{a\tau_n}|Y_{\tau_n}|^2-a\int_0^{\tau_n}e^{as}|Y_s|^2ds+2\int_0^{\tau_n}e^{as}Y_sf(s,Y_s,Z_s)ds
\\&&+\int_0^{\tau_n}e^{as}|g(s,Y_s,Z_s)|^2ds+2\int_0^{\tau_n}e^{as}Y_sdK_s\\&&+2\int_0^{\tau_n}e^{as}Y_sg(s,Y_s,Z_s)dB_s
-2\int_0^{\tau_n}e^{as}Y_sZ_sdW_s
\\&\leq&e^{a\tau_n}|Y_{\tau_n}|^2-a\int_0^{\tau_n}e^{as}|Y_s|^2ds\\&&+\frac{1}{\theta_1}\int_0^{\tau_n}e^{as}|Y_s|^2ds
+\theta_1\int_0^{\tau_n}e^{as}[4C^2(|Y_s|^2+|Z_s|^2)+2|f_s^0|^2]ds
\\&&
+(1+\theta_1)\int_0^{\tau_n}e^{as}\left(C|Y_s|^2+\alpha|Z_s|^2\right)ds+(1+\frac{1}{\theta_1})\int_0^{\tau_n}e^{as}|g_s^0|^2ds\\&&
+\frac{1}{\theta_2}\sup_{t\in[0,\tau_n]}e^{2at}|Y_t|^2+\theta_2|K_{\tau_n}|^2\\&&+2\int_0^{\tau_n}e^{as}Y_sg(s,Y_s,Z_s)dB_s
-2\int_0^{\tau_n}e^{as}Y_sZ_sdW_s.
\end{eqnarray*}

 On the other hand, from the equation
\begin{eqnarray*}
 K_{\tau_n}
=Y_0-Y_{\tau_n}-\int_0^{\tau_n}f(s,Y_s,Z_s)ds-\int_0^{\tau_n}g(s,Y_s,Z_s)dB_s+\int_0^{\tau_n}Z_sdW_s,
\end{eqnarray*}
we have
\begin{eqnarray*}\label{est:2}
|K_{\tau_n}|^2&\leq&
 d_1\left[|Y_0|^2+|Y_{\tau_n}|^2+(\int_0^{\tau_n}|f_s^0|ds)^2+\int_0^{\tau_n}(|Y_s|^2+|Z_s|^2)ds\right.\nonumber\\&&
\left.+|\int_0^{\tau_n}g(s,Y_s,Z_s)dB_s|^2+|\int_0^{\tau_n}Z_sdW_s|^2\right].
\end{eqnarray*}
Plugging this last inequality in the previous one to get
\begin{eqnarray*}
&&(1-d_1\theta_2)|Y_0|^2+(1-4C^2\theta_1-\alpha(1+\theta_1))\int_0^{\tau_n}e^{as}|Z_s|^2ds-d_1\theta_2\int_0^{\tau_n}|Z_s|^2ds
\\&\leq&(d_1\theta_2+e^{a\tau_n})|Y_{\tau_n}|^2+\left(\frac{1}{\theta_1}+4C^2\theta_1+C(1+\theta_1)-a\right)\int_0^{\tau_n}e^{as}|Y_s|^2ds
\\&&+2\theta_1\int_0^{\tau_n}e^{as}|f_s^0|^2ds+(1+\frac{1}{\theta_1})\int_0^{\tau_n}e^{as}|g_s^0|^2ds
+\frac{1}{\theta_2}\sup_{t\in[0,\tau_n]}e^{2at}|Y_t|^2\\&&
+d_1\theta_2\left[\int_0^{\tau_n}|f_s^0|^2ds+\int_0^{\tau_n}|Y_s|^2ds+|\int_0^{\tau_n}g(s,Y_s,Z_s)dB_s|^2+|\int_0^{\tau_n}Z_sdW_s|^2\right]
\\&&+2|\int_0^{\tau_n}e^{as}Y_sg(s,Y_s,Z_s)dB_s|
+2|\int_0^{\tau_n}e^{as}Y_sZ_sdW_s|.
\end{eqnarray*}
Choosing now $\theta_1,\theta_2$ small enough and $a>0$ such that $
\frac{1}{\theta_1}+4C^{2}\theta_1+C(1+\theta_1)-a<0, $ we obtain
\begin{eqnarray}\label{est:1}
\int_0^{\tau_n}|Z_s|^2ds
&\leq&d_2\left(\sup_{t\in[0,\tau_n]}|Y_t|^2+\int_0^{\tau_n}e^{as}|g_s^0|^2ds+\int_0^{\tau_n}e^{as}|f_s^0|^2ds
\right.\nonumber\\&&\left.+|\int_0^{\tau_n}e^{as}Y_sg(s,Y_s,Z_s)dB_s|+|\int_0^{\tau_n}e^{as}Y_sZ_sdW_s|\right.\nonumber
\\&&\left.+\theta_2|\int_0^{\tau_n}g(s,Y_s,Z_s)dB_s|^2+\theta_2
|\int_0^{\tau_n}Z_sdW_s|^2\right),
\end{eqnarray}
it follows that
\begin{eqnarray*}
&&{\bf E}\left(\int_0^{\tau_n}|Z_s|^2ds\right)^{\frac{p}{2}}
\nonumber\\&\leq&d_3{\bf
E}\left[\sup_{t\in[0,\tau_n]}|Y_t|^p+\left(\int_0^{\tau_n}|g_s^0|^2ds\right)^{\frac{p}{2}}+\left(\int_0^{\tau_n}|f_s^0|^2ds\right)^{\frac{p}{2}}
\right.\nonumber\\&&\left.+\left(
|\int_0^{\tau_n}e^{as}Y_sg(s,Y_s,Z_s)dB_s|\right)^{\frac{p}{2}}
+\left(
|\int_0^{\tau_n}e^{as}Y_sZ_sdW_s|\right)^{\frac{p}{2}}\right.\nonumber\\&&\left.
+\theta_2^{\frac{p}{2}}|\int_0^{\tau_n}g(s,Y_s,Z_s)dB_s|^p+\theta_2^{\frac{p}{2}}
|\int_0^{\tau_n}Z_sdW_s|^p\right].
\end{eqnarray*}

By the Burkh\"{o}lder-Davis-Gundy and Young's inequalities, we have
\begin{eqnarray*}\label{est:5}
&&{\bf
E}\left[\left|\int_0^{\tau_n}e^{as}Y_sg(s,Y_s,Z_s)dB_s\right|^{\frac{p}{2}}\right]
\nonumber\\&\leq& d_4{\bf
E}\left[\left(\int_0^{\tau_n}|Y_s|^2|g(s,Y_s,Z_s)|^2ds\right)^{\frac{p}{4}}\right]
\nonumber\\&\leq& d_4{\bf
E}\left[(\sup_{t\in[0,\tau_n]}|Y_t|)^{\frac{p}{2}}\left(\int_0^{\tau_n}|g(s,Y_s,Z_s)|^2ds\right)^{\frac{p}{4}}\right]
\nonumber\\&\leq& (\frac{d_4}{\theta_3}+\theta_3){\bf
E}[\sup_{t\in[0,\tau_n]}|Y_t| ^{p}]\nonumber\\&&+\theta_3{\bf
E}\left[\left(\int_0^{\tau_n}|g_s^0|^2ds\right)^{\frac{p}{2}}+\left(\int_0^{\tau_n}|Z_s|^2ds\right)^{\frac{p}{2}}\right]
\end{eqnarray*}
and
\begin{eqnarray*}
{\bf
E}\left[\left|\int_0^{\tau_n}e^{as}Y_sZ_sdW_s\right|^{\frac{p}{2}}\right]\leq
 \frac{d_5}{\theta_3} {\bf E}\left[\sup_{t\in[0,\tau_n]}|Y_t|
^{p}\right]+\theta_3{\bf
E}\left[\left(\int_0^{\tau_n}|Z_s|^2ds\right)^{\frac{p}{2}}\right].
\end{eqnarray*}

Plugging the two last inequalities in the previous one and using
 the Burkh\"{o}lder-Davis-Gundy inequality once again, it follows after
choosing $\theta_2$, $\theta_3 $ small enough (s.t. (\ref{est:1})
holds too):
\begin{eqnarray*}
&&{\bf E}\left[\left(\int_0^{\tau_n}
|Z_s|^2ds\right)^{\frac{p}{2}}\right]
\\&\leq& d\,{\bf E}\left[\sup_{t\in[0,\tau_n]}|Y_t|^p+\left(\int_0^{\tau_n}
|f_s^0|^2ds\right)^{\frac{p}{2}}+\left(\int_0^{\tau_n}
|g_s^0|^2ds\right)^{\frac{p}{2}}\right].
\end{eqnarray*}
Finally, we get the desired result by Fatou's Lemma.
\end{proof}
\begin{lemma}\label{lemma:2}
 Assume that (H1)-(H3) hold, let $(Y, Z, K)$ be a solution of the reflected
BDSDE (\ref{bsde:1}) where $Y\in\mathcal {S}^p$. Then there exists a
constant $d>0$ such that
\begin{eqnarray*}
&&{\bf E}\left[\sup_{t\in[0,T]}|Y_t|^p+\left(\int_0^T
|Z_s|^2ds\right)^{\frac{ p}{2}}+|K_T|^p\right] \\&\leq& d\,{\bf
E}\left[|\xi|^p+\left(\int_0^T|f_s^0|^2ds\right)^{\frac{p}{2}}
+\left(\int_0^T|g_s^0|^2ds\right)^{\frac{p}{2}}\right.\\&&\qquad\left.+\sup_{t\in[0,T]}(L_t^+)^p+\int_0^T|Y_s|^{p-2}{\bf
I}_{\{Y_s\neq 0\}}|g_s^0|^2ds \right].
\end{eqnarray*}
\end{lemma}
\begin{proof}
 From Lemma \ref{lemma:0}, for any $a\in {\bf R}$ and
any $0\leq t \leq T$, we have
\begin{eqnarray}\label{est:6}
&&e^{apt}|Y_t|^p+c(p)\int_t^Te^{aps}|Y_s|^{p-2}{\bf 1}_{\{Y_s\neq
0\}}|Z_s|^2ds\nonumber
\\&\leq&
e^{apT}|\xi|^p+p\int_t^Te^{aps}|Y_s|^{p-1}\widehat{Y_s}f(s, Y_s,
Z_s)ds+p\int_t^Te^{aps}|Y_s|^{p-1}\widehat{Y_s}dK_s\nonumber
\\&&
+c(p)\int_t^Te^{aps}|Y_s|^{p-2}{\bf 1}_{\{Y_s\neq
0\}}|g(s,Y_s,Z_s)|^2ds-ap\int_t^Te^{aps}|Y_s|^pds\nonumber
\\&&
+p\int_t^Te^{aps}|Y_s|^{p-1} \widehat{Y_s}
g(s,Y_s,Z_s)dB_s-p\int_t^Te^{aps}|Y_s|^{p-1}\widehat{Y_s}Z_sdW_s.
\end{eqnarray}
 By assumption (H2)
and  Young's inequality, we obtain
\begin{eqnarray}\label{est:7}
&&p\,{\bf E}\left[\int_t^Te^{aps}|Y_s|^{p-1}\widehat{Y_s}f(s, Y_s,
Z_s)ds\right]\nonumber
\\&\leq&{\bf E}[p\int_t^Te^{aps}|Y_s|^{p-1}|f_s^0|ds+Cp\int_t^Te^{aps}|Y_s|^{p-1}(|Y_s|+|Z_z|)ds]\nonumber
\\&\leq&(p-1)\theta_4^{\frac{p}{p-1}}{\bf E}\left(\sup_{s\in[0,T]}|Y_s|^p\right)
+\theta_4^{-p}{\bf
E}\left(\int_t^Te^{aps}|f_s^0|ds\right)^{p}\nonumber
\\&&+(Cp+\frac{p^2C^2}{2c(p)\theta_4}){\bf E}\left[\int_t^Te^{aps}|Y_s|^{p}ds\right]\nonumber
\\&&+\frac{c(p)}{2}\theta_4{\bf E}\left[\int_t^Te^{aps}|Y_s|^{p-2}{\bf 1}_{\{Y_s\neq
0\}}|Z_s|^2ds\right]
\end{eqnarray}
and
\begin{eqnarray}\label{est:8}
&&c(p)\,{\bf E}[\int_t^Te^{aps}|Y_s|^{p-2}{\bf 1}_{\{Y_s\neq
0\}}|g(s,Y_s,Z_s)|^2ds]\nonumber
\\&\leq&c(p)(1+\frac{1}{\theta_4}){\bf E}[\int_t^Te^{aps}|Y_s|^{p-2}{\bf 1}_{\{Y_s\neq
0\}}|g_s^0|^2ds]\nonumber\\&&+c(p)C(1+\theta_4){\bf
E}[\int_t^Te^{aps}|Y_s|^{p}ds]\nonumber\\&&+
c(p)\alpha(1+\theta_4){\bf E}[\int_t^Te^{aps}|Y_s|^{p-2}{\bf
I}_{\{Y_s\neq 0\}}|Z_s|^2ds].
\end{eqnarray}
Moreover, since $dK_s={\bf 1}_{\{Y_s\leq L_s\}}dK_s$,  we get from
Young's inequality
\begin{eqnarray*}
&&p\,{\bf E}[\int_t^Te^{aps}|Y_s|^{p-1}\widehat{Y_s}dK_s]
\\&\leq&
p\,{\bf E}[\int_t^Te^{aps}|Y_s|^{p-1}\widehat{Y_s}{\bf I}_{\{Y_s\leq
L_s\}}dK_s]
\\&\leq&p\,{\bf E}[\int_t^Te^{aps}|L_s|^{p-1}\widehat{L_s}dK_s]
\\&\leq&p\,{\bf E}[(\sup_{s\in[0,T]}L_s^+)^{p-1}\int_t^Te^{aps}dK_s]
\\&\leq&
\frac{p-1}{\theta_4^{\frac{p-1}{p}}}{\bf
E}[\sup_{s\in[0,T]}(L_s^+)^{p}]+\theta_4^p{\bf
E}(\int_t^Te^{aps}dK_s)^p
\\&\leq&
d_6\left[\theta_4^{-\frac{p-1}{p}}{\bf
E}(\sup_{s\in[0,T]}(L_s^+)^{p})+\theta_4^p\, {\bf E}|K_T|^p\right].
\end{eqnarray*}
On the other hand, by assumption (H2), the
Burkh\"{o}lder-Davis-Gundy inequality and Lemma 3.1, we have
\begin{eqnarray}\label{est:9}
 {\bf E}|K_T|^p \leq d_7
 {\bf E}\left[\sup_{s\in[0,T]}|Y_s|^p+\left(\int_0^T|f_s^0|ds\right)^p+\left(\int_0^T|g_s^0|ds\right)^p\right],
\end{eqnarray}
it follows that
\begin{eqnarray}\label{est:10}
 p\,{\bf E}[\int_t^Te^{aps}|Y_s|^{p-1}\widehat{Y_s}dK_s]\nonumber&\leq&
 d_8 {\bf E}\left[\theta_4^p\sup_{s\in[0,T]}|Y_s|^p+\theta_4^{-\frac{p-1}{p}}\sup_{s\in[0,T]}(L_s^+)^{p}\right.\\&&\left.
 +\theta_4^p\left(\int_0^T|f_s^0|ds\right)^p+\theta_4^p\left(\int_0^T|g_s^0|ds\right)^p\right].
\end{eqnarray}
Combining (\ref{est:7})-(\ref{est:9}), taking expectation on both
sides of (\ref{est:6}) to obtain
\begin{eqnarray*}
&&{\bf E}\left[e^{apt}|Y_t|^p+c(p)\int_t^Te^{aps}|Y_s|^{p-2}{\bf
I}_{\{Y_s\neq 0\}}|Z_s|^2ds\right]\nonumber
\\&\leq&
{\bf
E}\left\{e^{apT}|\xi|^p+d_8\theta_4^{-\frac{p-1}{p}}\sup_{s\in[0,T]}(L_s^+)^{p}
+\left[(p-1)\theta_4^{\frac{p}{p-1}}+d_8\theta_4^p\right]\sup_{s\in[0,T]}|Y_s|^p\right.\nonumber\\&&\left.
+\left[(Cp+\frac{p^2C^2}{2c(p)\theta_4})+c(p)C(1+\theta_4)-ap\right]\int_t^Te^{aps}|Y_s|^{p}ds\right.\nonumber\\&&\left.
+\left[\frac{c(p)}{2}\theta_4+c(p)\alpha(1+\theta_4)\right]\int_t^Te^{aps}|Y_s|^{p-2}{\bf
I}_{\{Y_s\neq 0\}}|Z_s|^2ds \right.\nonumber\\&&\left.
+\theta_4^{-p}\left(\int_t^Te^{aps}|f_s^0|ds\right)^{p}+d_8\theta_4^{p}
\left(\int_t^T|f_s^0|ds\right)^{p}+d_8\theta_4^{p}\left(\int_t^Te^{aps}|g_s^0|ds\right)^{p}
\right.\nonumber\\&&\left.
+\left[c(p)(1+\frac{1}{\theta_4})+d_8\theta_4^{p}\right]
\int_t^Te^{aps}|Y_s|^{p-2}{\bf 1}_{\{Y_s\neq
0\}}|g_s^0|^2ds\right\}.
\end{eqnarray*}
Choosing $\theta_4$ small enough and $a>0$ such that
\begin{eqnarray}\label{est:11}
(Cp+\frac{p^2C^2}{2c(p)\theta_4})+c(p)C(1+\theta_4)-ap<0,
\end{eqnarray}
 we get
\begin{eqnarray}\label{est:12}
&&{\bf E}\left[e^{apt}|Y_t|^p+c(p)\int_t^Te^{aps}|Y_s|^{p-2}{\bf
I}_{\{Y_s\neq 0\}}|Z_s|^2ds\right]\nonumber
\\&\leq&
d_9 {\bf
E}\left[|\xi|^p+\sup_{s\in[0,T]}(L_s^+)^{p}+\left(\int_t^T|f_s^0|^2ds\right)^{\frac{p}{2}}
+\left(\int_t^T|g_s^0|^2ds\right)^{\frac{p}{2}}\right.\nonumber\\&&\left.+\int_t^Te^{aps}|Y_s|^{p-2}{\bf
I}_{\{Y_s\neq 0\}}|g_s^0|^2ds\right]+d_9\theta_4^p\, {\bf
E}[\sup_{s\in[0,T]}|Y_s|^{p}].
\end{eqnarray}
Next using the Burkh\"{o}lder-Davis-Gundy inequality we have
\begin{eqnarray}\label{est:13}
&&{\bf E}\left[\sup_{t\in[0,T]}\left
|p\int_0^Te^{aps}|Y_s|^{p-1}\widehat{Y}_sZ_sdW_s\right|\right]\nonumber\\&\leq&
\frac{1}{4}{\bf E}[\sup_{t\in[0,T]}e^{apt}|Y_t|^{p}]+d_{10}{\bf E}
\left(\int_0^Te^{aps}|Y_s|^{p-2}{\bf 1}_{\{Y_s\neq
0\}}|Z_s|^2ds\right)
\end{eqnarray}
and
\begin{eqnarray}\label{est:14}
&&{\bf
E}\left[\sup_{t\in[0,T]}\left|\int_0^Te^{aps}|Y_s|^{p-1}\widehat{Y}_sg(s,Y_s,Z_s)dB_s\right|\right]\nonumber\\&\leq&
\frac{1}{4}{\bf
E}\left[\sup_{t\in[0,T]}e^{apt}|Y_t|^{p}\right]+d_{11}{\bf E}
\int_0^Te^{aps}|Y_s|^{p-2}{\bf 1}_{\{Y_s\neq
0\}}|g(s,Y_s,Z_s)|^2ds\nonumber
\\&\leq&
\frac{1}{4}{\bf E}[\sup_{t\in[0,T]}e^{apt}|Y_t|^{p}]+ d_{11}{\bf E}
\left[\int_0^Te^{aps}|Y_s|^{p-2}{\bf 1}_{\{Y_s\neq
0\}}|g_s^0|^2ds\right.\nonumber\\&&
\left.+\int_0^Te^{aps}|Y_s|^{p}ds+\int_0^Te^{aps}|Y_s|^{p-2}{\bf
I}_{\{Y_s\neq 0\}}|Z_s|^2ds\right].
\end{eqnarray}
Next going back to (\ref{est:6}), using the
Burkh\"{o}lder-Davis-Gundy inequality together with the inequalities
(\ref{est:12})-(\ref{est:14}), we get after choosing $\theta_4$
small enough (s.t. inequality (\ref{est:11}) holds too)
\begin{eqnarray*}
&& {\bf
E}\left[\sup_{t\in[0,T]}e^{apt}|Y_t|^p+\int_t^Te^{aps}|Y_s|^{p-2}{\bf
I}_{\{Y_s\neq 0\}}|Z_s|^2ds\right] \\&\leq& d\,{\bf
E}\left[|\xi|^p+(\int_0^T|f_s^0|^2ds)^{\frac{p}{2}}
+(\int_0^T|g_s^0|^2ds)^{\frac{p}{2}}\right.\\&&\left.
+\sup_{t\in[0,T]}(L_t^+)^p+\int_0^Te^{aps}|Y_s|^{p-2}{\bf
I}_{\{Y_s\neq 0\}}|g_s^0|^2ds\right].
\end{eqnarray*}
We then complete the proof  by the inequality (\ref{est:9}).
\end{proof}

\begin{lemma}\label{lemma:3}
Let  $(Y',Z',K')$ and $(Y,Z,K)$  be  the solution of the reflected
BDSDE (\ref{bsde:1}) associated with  $(\xi', f', g', L)$ and $(\xi,
f, g, L)$ respectively, where $(\xi', f', g', L)$ and $(\xi, f, g,
L)$ satisfy assumptions (H1)-(H3). Then
\begin{eqnarray*}
&&{\bf E}\left[\sup_{t\in[0,T]}|Y'_t-Y_t|^p+\left(\int_0^T
|Z'_s-Z_s|^2ds\right)^{\frac{ p}{2}}\right]
\\&\leq&
d\,{\bf E}\left[|\xi'-\xi|^p+(\int_0^T|f'(s, Y_s, Z_s)-f(s,
Y_s,Z_s)|^2ds)^{\frac{p}{2}}\right.\\&&\left.\qquad+(\int_0^T|g'(s,
Y_s,Z_s)-g(s, Y_s,Z_s))|^2ds)^{\frac{p}{2}}\right].
\end{eqnarray*}
\end{lemma}
\begin{proof}   The proof of the lemma is a combination of the proofs
of Lemmas 3.1 and 3.2 with a slight change. Indeed, let
 $$\overline{\xi}=:\xi'-\xi, \quad(\overline{Y},\overline{Z},\overline{K})=:(Y'-Y,Z'-Z,K'-K).$$
One can  easily to check that
$(\overline{Y},\overline{Z},\overline{K})$ is a solution to the
following BDSDE:
\begin{eqnarray*}
\overline{Y}_t=\overline{\xi}+\int_t^Th(s,\overline{Y}_s,
\overline{Z}_s)ds+\int_t^Tk(s,\overline{Y}_s,
\overline{Z}_s)dB_s+\overline{K}_T-\overline{K}_t-\int_t^T\overline{Z}_sdW_s,
\end{eqnarray*}
where
\begin{eqnarray*}
 h(s,y,z)=:f'(s, y+Y_s,z+Z_s)-f(s,Y_s,Z_s),\\
 k(s,y,z)=:g'(s, y+Y_s,z+Z_s)-g(s,Y_s,Z_s).
 \end{eqnarray*}
 Obviously, the functions $h$ and $k$ are Lipschitz w.r.t
 $(y,z)$.

 Let's note that
 \begin{eqnarray*}
 \int_0^te^{aps}\overline{Y}_sd\overline{K}_s=-\int_0^te^{aps}(Y'_s-L_s)dK_s-\int_0^te^{aps}(Y_s-L_s)dK'_s\leq0
 \end{eqnarray*}
 and
 \begin{eqnarray*}
 \int_0^te^{aps}|\overline{Y}_s|^{p-1}\widehat{\overline{Y}_s}d\overline{K}_s
 &=&-\int_0^te^{aps}|\overline{Y}_s|^{p-2}{\bf 1}_{\{\overline{Y}_s\neq
0\}}(Y'_s-L_s)dK_s\\&&-\int_0^te^{aps}|\overline{Y}_s|^{p-2}{\bf
I}_{\{\overline{Y}_s\neq 0\}}(Y_s-L_s)dK'_s\\&\leq&0.
 \end{eqnarray*}
The rest of the proof follows It\^{o}'s formula, Lemma 2.1 and the
steps similar to those in the proofs  of Lemmas 3.1 and
3.2.\end{proof}

\subsection{Existence and uniqueness of a solution}

In order to obtain the existence and uniqueness result, we make the
following supplementary assumption:
\\
(H4) $g(\cdot, 0,0)\equiv 0$.

The following result due to Bahlali et al. (2009).
\begin{lemma}\label{lemma:4}
Let $p=2$. Assume that (H1)-(H3) hold. Then the reflected BDSDE
(\ref{bsde:1}) has a unique solution $(Y,Z,K)\in\mathcal {S}^2\times
\mathcal{M}^2_d\times \mathcal {S}^2_{ci}$.
\end{lemma}

We now state and prove our main result.

\begin{theorem}\label{theorem:1}
Assume (H1)-(H4), then the reflected BDSDE (\ref{bsde:1}) has a
unique solution $(Y,Z,K)\in\mathcal {S}^p\times
\mathcal{M}^p_d\times \mathcal {S}^p_{ci}$.
\end{theorem}
\begin{proof}  The uniqueness  is an immediate consequence of Lemma
\ref{lemma:3}. We next to prove the existence.

 For each $n, m\in \mathbb{N}^{\ast}$, define
$$
\xi_n=q_n(\xi), f_n(t,x,y)=f(t,x,y)-f_t^0+q_n(f_t^0),
L^m_t=q_m(L_t),
$$
where $q_k(x)=x\frac{k}{|x|\vee k}$. One can easily to check that
the items $\xi_n, f_n$ and $L^m$ satisfy the assumptions (H1)-(H3),
it follows from Lemma \ref{lemma:4} that, for each $n, m\in
\mathbb{N}^{\ast}$, there exists a unique solution $(Y^n,Z^n,K^n)\in
L^2$ for the reflected BDSDE associated with $(\xi_n, f_n, g, L^m)$,
but in fact also in $L^p$,
according assumption (H4) and the Lemmas \ref{lemma:1} and \ref{lemma:2}. %%%%%%%%%%%%%%%

  %%%%%%%%%%%%%%

Next, from Lemma \ref{lemma:3},  for
$(i,n)\in\mathbb{N}\times\mathbb{N}^{\ast}$, we have
\begin{eqnarray*}
&&{\bf E}\left\{\sup_{t\in[0,T]}|Y_t^{n+i}-Y_t^n|^p+\left(\int_0^T
|Z_s^{n+i}-Z_s^n|^2ds\right)^{\frac{p}{2}}\right\}
\\&\leq& d{\bf E}\left\{|\xi_{n+i}-\xi_n|^p+\left(\int_0^T
|q_{n+i}(f_s^0)-q_n(f_s^0)|^2ds\right)^{\frac{p}{2}}\right\}.
\end{eqnarray*}
Clearly, the right side  of  above inequality tend to 0 as
$n\rightarrow \infty$, uniformly on $i$ so that $(Y^n,Z^n)$ is a
Cauchy sequence in $\mathcal {S}^p\times \mathcal{M}^p_d$. Let's
denote by $(Y,Z)\in\mathcal {S}^p\times \mathcal{M}^p_d$ it limit.
By the equation
\begin{eqnarray*}
K_t^n=Y_0^n-Y_t^n-\int_0^tf_n(s,Y_s^n,Z_s^n)ds-\int_0^tg(s,Y_s^n,Z_s^n)dB_s+\int_0^tZ_s^ndW_s,
\end{eqnarray*}
 similar computation can derive that $(K_t^n)_{n\geq 1}$ is also a Cauchy
 sequence in $\mathcal {S}^p_{ci}$, then there
 exists a non-decreasing process $K_t\in \mathcal {S}^p_{ci}\;(K_0=0)$ such
 that
\begin{eqnarray*}
{\bf E}(|K_t^n-K_t|^p)\rightarrow 0,\;\mbox{as}\; n \rightarrow
\infty
\end{eqnarray*}
and
\begin{eqnarray*}
\int_0^T(Y_s-L_s^m)dK_s=0.
\end{eqnarray*}
 By the dominated convergence theorem, we then get
\begin{eqnarray*}
\int_0^T(Y_s-L_s^m)dK_s\rightarrow\int_0^T(Y_s-L_s)dK_s,\;\mbox{as}\;
m \rightarrow \infty.
\end{eqnarray*}
It follows that the limit $(Y,Z,K)$ is a $L^p$-solution of reflected
BDSDE with  $(\xi, f,g,L)$. The proof is complete.
\end{proof}

\end{document}